\documentclass{lms}

\newnumbered{conjecture}{Conjecture} 
\newnumbered{definition}{Definition}

\newtheorem{theorem}{Theorem}[section]
\newtheorem{lemma}[theorem]{Lemma}
\newtheorem{corollary}[theorem]{Corollary}

\newcommand\np{{\mbox{NP}}}
\newcommand\nee{{\mbox{NE}}}
\newcommand\ee{{\mbox{E}}}
\newcommand\pp{{\mbox{P}}}
\newcommand\npco{{\np \cap \mbox{co}\np}}
\newcommand\neco{{\nee \cap \mbox{co}\nee}}
\newcommand\npcpp{{(\npco)/\mbox{poly}}}

\newcommand\nw{{{NW}_{A_n,f}}}
\newcommand\tpv{\mbox{T}_{PV}}
\newcommand\ufla{\Psi_{L,{\cal F}, {\cal A}, k_0}}
\newcommand\sat{{\sf sat}_k}
\newcommand\prov{{{\sf prov}_Q^{k^c}}}
\newcommand\refk{{{\sf ref}_{Q(P)}^k}}
\newcommand\tk{{\mbox{TAUT}_k}}
\newcommand\lbf{{\mbox{LB}_P(c)}}

\newcommand\nn{{\{0,1\}^n}}
\newcommand\mm{{\{0,1\}^m}}
\newcommand\kk{{\{0,1\}^k}}
\newcommand\el{{\{0,1\}^{\ell}}}
\newcommand\uu{{\{0,1\}^*}}

\newcommand\prob{{\mbox Prob}}

\title[Finding hard tautologies]%
{On the computational complexity of finding hard tautologies}

\author{Jan Kraj\'{\i}\v{c}ek}

\classno{03F20 (primary), 68Q15 (secondary).}

\extraline{Supported
in part by grant IAA100190902. A part of this research has been done
while the author 
was a visiting fellow of the Isaac Newton Institute 
in Cambridge (programme {\em Semantics and Syntax})
in Spring 2012.
Also partially affiliated with 
the Institute of Mathematics of the Academy of Sciences.}

\date{Faculty of Mathematics and Physics\\
Charles University in Prague}

\begin{document}
\maketitle

\begin{abstract}
It is well-known (cf. K.-Pudl\'ak\cite{KP-jsl})
that a polynomial time algorithm finding tautologies hard
for a propositional proof system $P$ exists iff $P$ is not optimal.
Such an algorithm takes $1^{(k)}$ and outputs a tautology $\tau_k$
of size at least $k$ such that $P$ is not p-bounded on the set of all
$\tau_k$'s.

We consider two more general search problems involving finding
a hard formula, {\bf Cert} and {\bf Find}, motivated by two hypothetical situations:
that one can prove that $\np \neq co\np$ and that no optimal proof system exists.
In {\bf Cert} one is asked to find a witness that a given non-deterministic
circuit with $k$ inputs does not define $TAUT \cap \kk$. In {\bf Find},
given $1^{(k)}$ and a tautology $\alpha$ of size at most $k^{c_0}$, one should
output a size $k$ tautology $\beta$ that has no size $k^{c_1}$ $P$-proof
from substitution instances of $\alpha$.

We shall prove, assuming the existence of an exponentially hard one-way permutation,
that {\bf Cert} cannot be solved by a time $2^{O(k)}$ algorithm. Using a
stronger hypothesis about the proof complexity of Nisan-Wigderson generator
we show that both problems {\bf Cert} and {\bf Find} are actually only partially defined
for infinitely many $k$ (i.e. there are inputs corresponding to $k$ for which
the problem has no solution).
The results are based on interpreting the Nisan-Wigderson generator
as a proof system.

\end{abstract}

A {\bf propositional proof system} in the sense of Cook and Reckhow \cite{CR}
is a polynomial time relation $P(x,y)$ such that for a binary string $\tau$:
\[
\tau \in \mbox{TAUT}\ \ \mbox{iff }\ \ \exists \pi \in \uu P(\tau, \pi)
\]
where TAUT is the set of propositional tautologies (in DeMorgan
language for the definiteness). 
Any string $\pi$ for which $P(\tau, \pi)$ holds is called 
a {\bf $P$-proof} of $\tau$.
A proof system (tacitly propositional from now on) is {\bf p-bounded} iff
there exists a constant $c \geq 1$ such that the above holds even
with the requirement that $|\pi| \le |\tau|^c$. 
Cook and Reckhow \cite{CR} noted that a p-bounded
proof system exists iff $\np = \mbox{co}\np$. Hence proving that
no p-bounded proof system exists would imply 
$\np \neq \mbox{co}\np$ and thus also $\pp \neq \np$. This fact elevated 
the investigation of lengths of proofs into a fundamental topic of mathematical
logic approach to computational complexity.

Strong lower bounds were proved for a variety of proof systems and several different
methods for this purpose were invented. Examples of proof systems that appear 
to be outside of the scope of current methods are the so called {\bf Frege systems}: the
usual text-book propositional calculi based on a finite number of axioms schemes and
inference rules (only quadratic lower bounds are known for them, cf.\cite{Kra-speed}).
This apparent failure could cause an uninformed reader to dismiss the whole area of proof complexity.
However, although we may not be near proving that $\np \neq \mbox{co}\np$,
the lower bounds for weaker proof systems proved so far do have consequences
interesting in their own right. For example, a single lower bound for a proof system
$P$ implies time lower bounds for a class $Alg(P)$ of SAT algorithms associated 
with $P$ and all commonly used SAT algorithms belong to such a  class for some $P$ for which we have
an exponential lower bound (cf.\cite{Kra-sat} and references given there).
Another type of consequences can be found 
in bounded arithmetic: a lengths-of-proofs lower bound
for $P$ often implies the unprovability of a true $\Pi^0_1$ sentence in a first-order theory
$T_P$ associated with $P$. These unprovability arguments do not use G\"{o}del's theorem
and the $\Pi^0_1$ sentences involved have typically a clear combinatorial
meaning. 
And last but not least, any 
super-polynomial lower bound for $P$ also implies that $\pp \neq \np$ is 
consistent with $T_P$.
We shall not survey these proof complexity topics in detail here and instead refer the reader
to expositions in \cite{kniha,Kra-methods,Kra-ecm,Pud-survey}
or in \cite[Chpt.27]{k2}.

\bigskip

In this paper we are interested in the question how hard it is - to be measured 
in terms of computational complexity here - 
to find tautologies hard (i.e. requiring long proofs) for a given proof system.
Proposing plausible candidates for tautologies hard for Frege systems and for
stronger proof systems turned out to be a quite delicate issue. 
The lack of a variety of good candidates is 
one of principal obstacles for proving lower bounds for strong systems. Of course,
in principle one would be happy to accept a suggestion for such a hard tautology
from a friendly oracle. However, the
experience with known lower bound proofs shows that it is essential to have explicit formulas
with a transparent combinatorial meaning. 

In particular, all first super-polynomial lower bounds for
proof systems for which we have any such bounds 
were proved for some sequence of
tautologies $\{\tau_k\}_k$ of size $|\tau_k| \geq k$ and constructible 
in polynomial time (or even log space) from $1^{(k)}$. This type of sequences of hard tautologies
has been considered in \cite{KP-jsl} and \cite[Chpt.14]{kniha} and it exists for a proof system
$P$ iff $P$ is not {\bf optimal}, i.e. there exists a proof system $Q$ that has a super-polynomial speed-up over
$P$ (w.r.t. lengths of proofs) on an infinite set of formulas.
It is consistent with the present knowledge, and indeed most researchers seem to conjecture that, 
that no optimal proof system exist and hence that for
each $P$ a p-time constructible sequence of hard formulas exist. However, deriving the existence of
hard formulas from the assumption of non-optimality 
is not very illuminating: it is a basic proof complexity result that $P$ cannot admit polynomial size
proofs of formulas expressing the soundness of a proof system $Q$ (these formulas are log space
constructible) if $Q$ has a super-polynomial speed-up
over $P$ on an infinite set of formulas (cf.\cite[Chpt.14]{kniha}). 
Hence, in a sense, deriving the existence of a polynomial time 
sequence of hard tautologies from the non-optimality assumption
amounts just to restating the assumption
in a different terminology.
We refer the reader to Beyersdorff-Sadowski\cite{BeySad} for further information 
and up-to-date references.

\bigskip

We shall consider in this paper two more general
search problems in which the task includes a 
requirement to find a hard tautology. 
The two problems model in their ways two hypothetical
situations: a situation when one can prove
$\np \neq co\np$ (i.e. super-polynomial lower bounds for all proof systems)
and a situation when one can prove that no optimal proof system exists
by having a uniform method how to construct from a given proof system
a stronger one. These two tasks, {\bf Cert} and {\bf Find}, will be defined
in Section \ref{tasks} (in Section \ref{pairs} we add one more search task
{\bf Pair} involving disjoint pairs of $\np$ sets).

We will prove (using the hypothesis of the existence of a hard
one-way permutation) that {\bf Cert} cannot be solved by exponential time
algorithms and (using a stronger 
hypothesis about the proof complexity of the
Nisan-Wigderson generator) that both {\bf Cert} and
{\bf Find} actually cannot be solved at all on infinitely many 
input lengths.
Our primary motivation for this research is to understand
what kind of consequences do various - both proven and conjectural - 
statements about the proof complexity of the Nisan-Wigderson generator have.
\bigskip

The paper is organized as follows. 
After the motivation and the definition of the search tasks {\bf Cert}
and {\bf Find} in Section \ref{tasks} we review some complexity theory
in Section \ref{1} and some proof complexity in 
Section \ref{2}.
The hardness results are proved in Sections \ref{4}
and \ref{6}, respectively (after a proof complexity interlude
in Section \ref{5}). The paper is concluded by Section \ref{pairs}
considering a related search task for disjoint pairs of sets and
a few remarks in Section \ref{rema}.

\medskip

We do assume only basic complexity theory and proof
complexity (e.g. the well-known relation between reflection principles 
and simulations).
But the reader may still benefit from understanding a
wider proof complexity context.
In particular, \cite[Chpt.27]{k2} overviews some fundamental problems
of proof complexity and
\cite[Chpts.29 and 30]{k2} survey\footnote{One can read these
chapters independently of the rest of the book.}  
the theory proof complexity generators (and list relevant literature).

\section{The search tasks {\bf Cert} and {\bf Find}}
\label{tasks}

We are going to consider two search tasks asking us to find formulas
with certain properties (and in Section \ref{pairs} we add one more).
Both are more complex than the mere task to construct hard tautologies for
a given proof system that was discussed in the introduction.
To motivate them we shall describe two thought situations in proof complexity; 
the search tasks are then abstract (and simplified) versions of those.

\medskip

First assume that you can prove (i.e. ZFC can) that $\np\ne co\np$ and
thus, in particular, super-polynomial lower bounds for all proof systems.
For a proof system $P$ and a constant $c \geq 1$
denote by $\lbf$ the statement
\[
\forall 1^{(k)} \exists \tau \ 
[|\tau|\geq k \wedge \tau \in TAUT \wedge 
\forall \pi (|\pi|\le |\tau|^c) \neg P(\tau, \pi)]
\]
formalizing a polynomial lower bound for $P$ with degree $c$.

It is easy to see that for any decent proof system (see
Section \ref{5} for a formal definition of decency), as long as we can
prove some specific polynomial lower bound we can also prove its
soundness. The decency assumption
allows to extend a proof of a falsifiable formula
$\varphi$ to a proof of $0$ and further to a proof of any $\tau$, all
in polynomial time. 

But by a simple application of G\"{o}del's theorem 
ZFC is not able to prove the soundness of all proof systems.
This suggests that we should be proving lower bounds conditioned
upon the assumption that $P$ is indeed a Cook-Reckhow proof system.
If $P$ were not complete we do not need to bother with lower bounds for it,
so the interesting clause of the Cook-Reckhow definition 
is the soundness and we are lead to implications:
\[
Ref_P \rightarrow \lbf 
\]
where $Ref_P$ is a universal sentence (in the language $L_{PV}$ of
Section \ref{2}) formalizing that any formula with a $P$-proof must
be a tautology. 

Now we simplify the situation a bit more.
Let $D(x,y)$ be a circuit in $k$ variables $x = (x_1,\dots, x_k)$
and $\ell = k^c$ variables $y = (y_1, \dots, y_{\ell})$
which we 
interpret as the provability relation of a proof system restricted to
formulas of size $k$ and proofs of size at most $\ell$. 
This motivates the following
{\bf search task} $\mbox{{\bf Cert}}(c)$ defined for any constant $c \geq 1$:

\begin{itemize}

\item input: $1^{(k)}$ and
a size $k^{c^2}$ circuit $D(x,y)$ in $k$ variables $x = (x_1,\dots, x_k)$
and $\ell = k^c$ variables $y = (y_1, \dots, y_{\ell})$

\item required output: either a size $k$ falsifiable formula $\varphi$ such that
$D(\varphi, y)$ is satisfiable or a size $k$ tautology $\tau$ such that
$D(\tau,y)$ is unsatisfiable.

\end{itemize}
The output
of $Cert(c)$ thus certifies that $D$ is not a non-deterministic circuit
(with input $x$ and non-deterministic variables $y$) that accepts $TAUT \cap \kk$.
In other words, the output either witnesses
that $D$ does not correspond to a sound proof
system or that the proof system admits a lower bound.

The provability relation of a proof system restricted to size $k$ formulas and size $\ell$
proofs can be computed by circuits of size $(k+\ell)^{O(1)}$. In the formulation
of the problem we have represented the $O(1)$ constant by $c$ as well.
In addition {\bf Cert} ignores the uniformity of such circuits corresponding
to a particular proof system (they can be constructed in log space from $1^{(k)}, 1^{(\ell)}$).
This is in line with the 
prevailing approach in complexity theory to reduce uniform 
problems to non-uniform finite combinatorial problems. Finally note that
in our simplification we are taking the reflection
principle just for the proof lengths corresponding to the lower bound we should witness;
this make sense due to the non-uniformity.

\bigskip

The second search task we shall consider is motivated by another thought
experiment. Assume that you can prove that no optimal proof system exists and, in fact,
that you have a uniform construction that from a proof system
$P$ produces a stronger proof system $Q(P)$ (i.e. not simulable by $P$).
For definiteness, assume that there is one oracle polynomial time machine
that for all $P$ defines $Q(P)$ when having the oracle for $P$.
Then we expect to be able to prove
\[
Ref_P\ \rightarrow\ Ref_{Q(P)}
\]
and, most importantly, that it is stronger
\[
Ref_P\ \rightarrow\ 
\forall 1^{(k)} \forall \pi (|\pi|\le k^c) \neg P(\refk, \pi)]
\]
where $\refk$ is a size $k^{O(1)}$ tautology formalizing
the soundness of $Q(P)$ w.r.t. all proof of size at most $k$
(we assume for simplicity that a proof is always at least as long as the
formula it proves so one parameter suffices).
See a similar formula in (\ref{m0}) in the proof of Lemma \ref{5.3}.

Any decent proof system can simulate $Q(P)$ if it can use $\refk$
as extra axioms (see Section \ref{5}). In the following problem
$\alpha$ represents a bit more\footnote{Really just a bit: for decent $P$ adding $\alpha$
is equivalent to adding the reflection principles for $P+\alpha$.}
generally any extra axiom.

Let $P$ be a proof system and $c_1 \geq c_0 \geq 1$ be 
constants. Consider the following 
promise {\bf computational task} $\mbox{{\bf Find}}(P,c_1,c_0)$:

\begin{itemize}

\item input: $1^{(k)}$ and a tautology $\alpha$ such that $|\alpha| \le k^{c_0}$

\item required output: any size $k$ tautology $\beta$ 
that has no proof in proof system $P + \alpha$, $P$ augmented by $\alpha$ as an extra
axiom scheme\footnote{The proof system $P + \alpha$ will be defined in Section \ref{5}.}, of size less 
than $k^{c_1}$.

\end{itemize}
The requirement that the size of $\beta$ is exactly $k$ is just for a technical
convenience; we could allow any interval $[k^{\Omega(1)}, k^{O(1)}]$ instead.

\section{Computational complexity preliminaries}
\label{1}

Let $n \rightarrow m = m(n)$ be an injective function such that $m(n) > n$
and let $f : \uu \rightarrow \uu$ be a Boolean function. The {\bf Nisan-Wigderson
generator} $NW_{A,f} : \nn \rightarrow \mm$ is defined using the notion of a 
design. A $(d,\ell)$-design on $[n]$ is a set system  
$A = \{J_i\subseteq [n]\}_{i\in [m]}$ on $[n] = \{1, \dots, n\}$
such that:
\begin{itemize}

\item $|J_i| = \ell$, for all $i$,

\item $|J_i \cap J_j| \le d$, for all $i \neq j$.

\end{itemize}
The $i$-th bit of $NW_{A,f}(x)$ is computed by $f_{\ell} : \el \rightarrow \{0,1\}$
from the $\ell$-bit string $x(J_i) := x_{j_1}\dots x_{j_{\ell}}$, where
\[
J_i\ =\ \{j_1< \dots < j_{\ell}\}\ 
\]
and $f_{\ell}$ is the restriction of $f$ to $\el$. In the future the parameter $\ell$
will be determined by $n$ and we shall denote the restriction $f$ as well.
Nisan and Wigderson\cite{NW} showed that there are such designs
for a wide range of parameters $n,m,d,\ell$ and that one can construct
them uniformly and feasibly. In particular,  
we can fix the parameters as follows:
\begin{equation}\label{e1}
\ell := n^{1/3}\ \ \mbox{ and }\ \ m := 2^{n^{\delta}}\ \ 
\mbox{ and }\ \ \ d := \log(m)\ ,
\end{equation}
where $1/3 \geq \delta > 0$ is arbitrary. We shall thus assume that fixing $n$ 
and $\delta$ fixes the other parameters and also some set system $A_n$ 
constructed from $1^{(n)}, 1^{(m)}$ in time $m^{O(1)}$ and
with parameters meeting the requirements. In fact, we need that
\begin{equation} \label{e2}
\mbox{{\em $J_i$ is computable from $i$  and $1^{(n)}$ in polynomial time.}}
\end{equation}
The design from \cite[L.2.5]{NW} has this property.

\bigskip

In our construction the function $f$ will be $\npco$.
By this we mean that
it is the characteristic function of a language in $\npco$.
Hence $f$ is defined by two $\np$ predicates 
\begin{equation}\label{e3}
\exists y (|y| \le |u|^{c} \wedge F_0(u,y))\ \ \mbox{ and }
\ \  
\exists y (|y| \le |u|^{c} \wedge F_1(u,y))
\end{equation}
with $F_0$ and $F_1$ polynomial-time relations and $c$ 
a constant such that
\begin{equation}\label{i0}
f(u) = a\ \mbox{ iff }\ \exists y (|y| \le |u|^{c} \wedge F_a(u,y))
\end{equation}
for $a = 0, 1$. Any string $y$ witnessing 
the existential quantifier will be called a witness for
$f(u)$. 

We shall use results from \cite{Kra-nwg} and those do assume that
$f$ has unique witnesses, meaning that
for each $u$ there is exactly one witness for $f(u)$.
A natural source of $\npco$ functions with unique witnesses are 
hard bits of one-way permutations.
That is, for a polynomial time (and intended to be one-way)
permutation $h : \uu \rightarrow \uu$ we have
\begin{equation}%\label{i1}
f(u)\ :=\ B(h^{(-1)}(u))\ 
\end{equation}
where $B(v)$ is a hard bit predicate for $h$.

The hardness of one-way permutations is measured as follows.
A polynomial time permutation $h$ is defined to be 
{\bf $\epsilon(\ell)$ one way with security parameter $t(\ell)$}
iff for all $\ell$ and any circuit $D$ with $\ell$ inputs
and of size at most $t(\ell)$ it holds:
\[
\prob_{v \in \{0,1\}^{\ell}}[ D(h(v)) = v] \ \le \ \epsilon(\ell)\ . 
\]
Using the Goldreich-Levin theorem we may assume that such a permutation
$h$ has a hard bit function $B(v)$. The details can be found in 
Goldreich \cite{Gol}.

\bigskip

Our construction needs to assume that $f$ is hard in the sense of
Nisan and Wigderson \cite{NW}. They
define $f$ to be $(\epsilon(\ell), S(\ell))$-hard if 
for every $\ell$ and every
circuit $C$ with $\ell$ inputs and of size at  most $S(\ell)$ it holds:
\[
\prob_{u \in \{0,1\}^{\ell}}[C(u) = f(u) ]\ < \ 
1/2 + \epsilon/2\ .
\]
They define then the {\bf hardness of $f$}, denoted {\bf $H_f(\ell)$},
to be the maximal $S$ such that the function is $(1/S, S)$-hard.
This simplification makes sense when $\epsilon$ has the rate about 
$m^{-O(1)}$  as in Nisan and Wigderson \cite{NW}. 

In the proof complexity
situations studied in \cite{Kra-nwg} 
the parameter $S$ plays the main role, with $\epsilon$ being primarily 
of the rate $\ell^{- O(1)}$.
This corresponds to the fact that in
applications of the original Nisan-Wigderson generators 
$m$ is usually exponentially large but for various purposes 
of proof complexity (especially lengths-of-proofs lower bounds)
the best choice would be at the opposite end: $m=n+1$. 
This lead in \cite{Kra-nwg} to keeping $\epsilon$ and $S$ separate
and using the notion of the approximating hardness (defined there)
in place of $H_f(\ell)$.
In this paper, however,  we shall use only those results from \cite{Kra-nwg}
where $m$ is exponentially large as in (\ref{e1})
and thus using the measure $H_f(\ell)$ suffices here. 

A one-way permutation $h$ with a hard bit $B$ is {\bf exponentially hard}
iff it is $2^{-{\ell}^{\Omega(1)}}$ one-way with security parameter 
$2^{{\ell}^{\Omega(1)}}$. The hardness $H_f(\ell)$ of $f$ is then 
$2^{{\ell}^{\Omega(1)}}$ as well.
Details of these constructions can be found in
Goldreich \cite{Gol}.

We will use in Sections \ref{4} and \ref{pairs}
the hypothesis that an exponentially hard one-way permutation
exists instead of the presumably weaker
assumption that an $\npco$ function $f$ with unique witnesses and with
exponential hardness $H_f$ exists. The only reason is that the former hypothesis
is more familiar than the latter one.

\section{Proof complexity preliminaries}
\label{2}

Although the formulation of the search tasks {\bf Cert}
and {\bf Find} may 
not suggests so
explicitly this investigation resulted from a research program
in proof complexity
about the so called proof complexity generators
and we shall use some ideas from this theory. 

We shall start with a 
proof complexity conjecture of Razborov\cite[Conjecture 2]{Raz03}. 
Take an arbitrary string $b \in \mm$ that is outside of the 
range $Rng(\nw)$ of $\nw$. 
The statement $b \notin Rng(\nw)$
is a $\mbox{co}\np$ property of $b$ and can be expressed by a propositional
formula $\tau(\nw)_b$ in the sense that
\[
\tau(\nw)_b \in \mbox{TAUT }\ \mbox{ iff }\ 
b \notin Rng(\nw)\ .
\]
The construction of the propositional translation
of the $\mbox{co}\np$ statement is analogous to the usual proof of
the NP-completeness of SAT. The details can be found in
any of \cite{Coo75,kniha,Pud-survey,k2}). 
Note that the size of the formulas is polynomial in $m$.
{\bf Razborov's conjecture} says that these tautologies are hard for Extended Frege system
EF for $\nw$ defined as above, with $m = 2^{n^{\Omega(1)}}$ and based on an
$\npco$ function $f$ that is hard on average for $\pp/\mbox{poly}$. 
Pich \cite{Pich} proved the conjecture for all proof
systems admitting feasible interpolation in place of EF.

In \cite{Kra-nwg} we have considered a generalization of the conjecture.
We shall recall only one part of that generalization
dealing with exponentially large $m$; in the other parts
$m = n+1$ and they use the
notion of approximating hardness of a function mentioned in the previous
section.

\begin{conjecture}[Part 3 of Statement (S) of \cite{Kra-nwg}]\label{2.1}
\hfill

Assume $f$ is an $\np \cap co\np$ function with unique witnesses
that has an exponential Nisan-Wigderson hardness 
$H_f(\ell) = 2^{\ell^{\Omega(1)}}$.

Then there is $\delta > 0$ such that for
$m(n) = 2^{n^{\delta}}$ and
for any infinite $\np$ set $R$
that has infinitely many elements whose length
equals to $m(n)$ for $n\geq 1$
it holds:
\[
Rng(\nw) \cap R \ \neq \ \emptyset \ .
\]
\end{conjecture}

Let us observe that Conjecture \ref{2.1} has a proof complexity corollary
including Razborov's conjecture.

\begin{lemma}\label{2.2}
\hfill

Let $P$ be any proof system.
Assume that Conjecture \ref{2.1} holds and that 
the Nisan-Wigderson hardness $H_f(\ell)$ of an
$\npco$ function $f$ with unique witnesses 
is $2^{\ell^{\Omega(1)}}$.

Then there exists $\delta > 0$ such that 
for all $c \geq 1$, the size of $P$-proofs of formulas $\tau(\nw)_b$ for 
all large enough $n$ and all $b \notin Rng(\nw)$ of size 
$|b| = m(n)$
is bigger than $|\tau(\nw)_b|^c$.

\end{lemma}

\begin{proof}

Note that the set $R$ of all $b$ 
of lengths $m(n)$ for $n \geq 1$
for which $\tau(\nw)_b$ has a $P$-proof
of size at most  $|\tau(\nw)_b|^c$ is in $\np$.

\end{proof}

Now we recall (a part of) the consistency result 
from \cite{Kra-nwg} concerning Conjecture \ref{2.1}. 
Its technical heart is a lower bound on complexity of
functions solving a certain search task associated with
$\nw$ and that would, in principle, suffice for our purposes here.
Using the consistency result itself, however, seems to decrease the
number of technicalities one otherwise needs to discuss. 

We first recapitulate a few basic definitions.
Cook \cite{Coo75} has defined a theory PV whose 
language $L_{PV}$ has a name for every 
polynomial-time algorithm
obtained from a few basic algorithms by the composition and by the limited
recursion on notation, following Cobham's \cite{Cob64} characterization
of polynomial time. The details of the definition of $L_{PV}$ can be found in
\cite{Coo75,kniha} but are not important here. 
In fact, neither is the theory PV itself
as we shall work with  
the true universal first-order theory of ${\mathbf N}$ in the 
language $L_{PV}$. We shall denote this theory 
$\tpv$, as in \cite{Kra-nwg}.
Note that $\tpv$ contains formulas expressing
the soundness of all proof systems.

Let $f$ be an $\np \cap \mbox{co}\np$ function defined as in (\ref{i0}). 
Let us abbreviate by $G(w,z,x,y)$ the open $L_{PV}$
formula
\begin{equation}%\label{g1}
(z_x=0 \wedge F_0(w(J_x), y)) \vee (z_x=1 \wedge F_1(w(J_x), y))
\end{equation}
where $J_x$ is from the set system $A_n$ (polynomial time definable from $1^{(n)}$ and $x$)
and $F_0$, $F_1$ are 
from (\ref{e3}).

We do not have a symbol in $L_{PV}$
for the function on $\uu$ computed 
for $n \geq 1$ on $\nn$ by
$\nw$ as it is not a polynomial time function,
and the function 
has to be defined. One possible formalization of the 
statement $\nw(w)= z$ for $|w|=n$ and $|z|=m$ is then
\begin{equation}%\label{h0}
\forall x \in [m] \exists y (|y| \le \ell^c)\ G(w,z,x, y)\ 
\end{equation}
with $c$ from (\ref{e3}). 
Now we are ready to state the result from \cite{Kra-nwg}
we shall need.

\begin{theorem}[Kraj\'{\i}\v{c}ek{\cite[Thm.4.2(part 3)]{Kra-nwg}}]
\label{3.1}
\hfill

Assume $f$ is an 
$\np \cap co\np$ function with unique witnesses
having the Nisan-Wigderson hardness $H_f(\ell)$
at least $2^{\ell^{\Omega(1)}}$.

Then there is $\delta > 0$ such that for any 
$\np$ set $R$ 
that has infinitely many elements whose length
equals to $m(n)$ for $n\geq 1$ and
defined by $L_{PV}$ formula
$\exists v (|v| \le |z|^d) R_0(z,v)$, with $R_0$ open,
theory $\tpv$ does not prove the universal closure of the formula
\[
A\ \rightarrow\ B
\]
where $A$ is the formula with variables $v, w, z, n, m, \ell$
\[
n = |w| \wedge m = |z| \wedge m = 2^{n^{\delta}} \wedge \ell = n^{1/3}
\wedge |v| \le m^d \wedge R_0(z,v)
\]
and $B$ is the formula
\[
\exists x \in [m] \forall y (|y| \le \ell^c)
\ \neg G(w, z, x, y)\ .
\]
\end{theorem}

This statement is in \cite{Kra-nwg} derived from a bit 
finer model-theoretic
result. 

\section{The hardness of task {\bf Cert}}
\label{4}

The argument we shall use to derive the hardness of {\bf Cert}
applies to a more general situation which we describe now.

By an ${\npcpp}$ {\bf algorithm} we shall mean two polynomial time 
predicates $F_0(x,y,z)$ and $F_1(x,y,z)$ and a constant $c \geq 1$
similarly as in (\ref{e3}) but now with an extra argument $z$ for the non-uniform
advice, and a sequence of advice strings $\{w_k\}_k$ such that $|w_k| \le k^c$
(w.l.o.g. we use constant $c$ also in the bound to the length of advice strings). 
We shall assume that
\begin{equation} \label{h1}
\forall x, z (|z|\le |x|^c)\ 
[\exists y (|y| \le |x|^c) F_0(x,y,z)] \oplus [\exists y (|y| \le |x|^c) F_1(x,y,z)]
\end{equation}
is valid where $\oplus$ is the exclusive disjunction. 
Thus an $\npcpp$ algorithm is an $\npco$ set
of pairs $(x, z)$ of appropriate lengths
augmented by a sequence of advice strings substituted for $z$.
In our situation it is more natural to talk about algorithms than sets as
we shall be looking for "errors they make".
We shall denote such an algorithm $({\cal F}, \{w_k\}_k)$ where
${\cal F}$ is the triple $(F_0, F_1, c)$ from (\ref{h1}).

For $L$ a language let us denote
by $L_k$ the truth table of the characteristic function of $L$ on $\kk$.
If $L \in \neco$ then the set $R^L$ of such strings $\{L_k\ |\ k \geq 1\}$
is in $\np$ and can be defined by an $L_{PV}$ formula as
\begin{equation}%\label{j1}
z \in R^L\ \mbox{ iff }\ 
\exists v (|v| \le |z|^d) R^L_0(z,v)
\end{equation}
with $R^L_0$ an open formula. Any $v$ witnessing the existential quantifier for 
$z$ will be called a witness for $z \in R^L$.
Note that $TAUT \in \neco$.

For a language $L \in \neco$ and a triple $\cal F$ as in (\ref{h1}) define the
{\bf search task $\mbox{{\bf Err}}(L, {\cal F})$} as follows:

\begin{itemize}

\item input: $1^{(k)}$, string $L_k$ and a witness $v$ for $L_k \in R^L$,
and a string $w$ such that $|w| \le k^c$

\item required output: a string $x \in \kk$ such that $\cal F$
using $w$ as an advice string makes an error on $x$:
\begin{equation}%\label{j2}
\forall y (|y|\le |x|^c) \  
[(x \in L_k \wedge \neg F_1(x,y,w)) \vee
(x \notin L_k \wedge \neg F_0(x,y,w))]
\end{equation}
\end{itemize}

\begin{theorem} \label{4.1}
Assume that an exponentially hard one-way permutation exists.
Let $L$ be a language such that 
$L \in \neco$.

Then there exists a triple ${\cal F}$ 
as in (\ref{h1}) such that no deterministic 
polynomial time algorithm can solve $Err(L,{\cal F})$
on all inputs for all sufficiently large lengths $k$.
\end{theorem}

\begin{proof}

Assume that language $L$ satisfies the hypothesis of the theorem
and let $\cal F$ be any triple as in (\ref{h1}).
Assume that $\cal A$ is a deterministic polynomial time algorithm that attempts to
solve $Err(L, {\cal F})$ on all inputs for all $k \geq k_0$, for some $k_0 \geq 1$.

We are going to define a universal $L_{PV}$ sentence 
\[
\ufla
\]
that is true iff 
$\cal A$ solves $Err(L, {\cal F})$ for all inputs for all $k \geq k_0$.

The sentence
$\ufla$ is the universal closure of:
\[
C\ \rightarrow \ D
\]
where $C$ is the formula
\[
|z|=2^k \wedge |v|\le |z|^d \wedge R^L_0(z,v)
 \wedge k \geq k_0 \wedge |w| \le k^c \wedge 
x = {\cal A}(1^{(k)},z,v,w)
\]
with $\cal A$ represented by an $L_{PV}$ function symbol,
and $D$ is the formula
\[
|x| = k \ \wedge \ 
\forall y (|y|\le |x|^c)
[(x \in z \wedge \neg F_1(x,y,w)) \vee
(x \notin z \wedge \neg F_0(x,y,w))]\ .
\]
The following should be obvious:

\medskip
\noindent
{\bf Claim 1:}{\em 
Algorithm $\cal A$ solves $Err(L, {\cal F})$ for all
inputs for all $k \geq k_0$ iff the sentence 
$\ufla$ is true.
}

\medskip
\noindent
We are going now to define a specific $\npcpp$ triple ${\cal F}$
as in (\ref{h1}) 
such that $\ufla$ will be false for all $L \in \neco$,
all polynomial time algorithms $\cal A$ and all $k_0 \geq 1$.

\bigskip

Let $h$ be an exponentially hard one-way permutation with a hard bit $B$. Hence 
the function $f$ from (\ref{i0}) has the exponential hardness
$H_f(\ell) = 2^{- \ell^{\Omega(1)}}$.
The existence of such $h$ and $B$ is 
guaranteed by the hypothesis of the theorem.

Take $\delta > 0$ provided by Theorem \ref{3.1} and put $k:= n^\delta$. Using
$\nw : \nn \rightarrow \mm = \{0,1\}^{2^k}$ define 
an $\npcpp$ triple ${\cal F} = (F_0, F_1, c)$ as n (\ref{h1})
as follows: put 
$c := \delta^{-1}$ and for 
$x \in \kk$, $y \in \nn = \{0,1\}^{k^c}$, $w$ of size $|z| = n$ and $a = 0,1$ define:
\begin{equation}\label{h2}
F_a(x,y,w)\ :=\ 
[h(y) = w(J_x) \wedge B(y) = a]\ .
\end{equation}
In other words, on input $x$ the algorithm computes the $x$-th bit
of $\nw(w)$.

\medskip
\noindent
{\bf Claim 2:} {\em 
For no $L \in \neco$,
no polynomial time algorithm $\cal A$ and no $k_0 \geq 1$ is the
sentence $\ufla$ true.
}

\medskip

To see this note that by substituting term ${\cal A}(1^{(k)},z,v,w)$ for $x$
in $C \rightarrow D$ and quantifying it existentially $\exists x (x \in [m])$
allows to deduce from $\ufla$ the universal closure of $A \rightarrow B$
from Theorem \ref{3.1}. Hence, by that theorem, $\ufla$ cannot be true.

\bigskip

Claims 1 and 2 imply the theorem.

\end{proof}

We remark that the argument can be actually extended to rule out
a larger class of algorithms $\cal A$: the so called Student - Teacher
interactive computations of \cite{KPT,KPS} (see also \cite{kniha}).

\bigskip

Let $({\cal F}, \{w_k\}_k)$ be as above. Define circuits $D_k(x,y)$ to be 
(some canonical) circuits with $k$ inputs $x$ and $k^c$ inputs $y$
that outputs $1$ iff $F_1(x,y,w_k)$ holds. We can choose $c \geq 1$ large enough
so that $D_k$ has size at most $k^{c^2}$.

Given $1^{(k)}$, a time $2^{O(k)}$ algorithm can compute the string $\tk$ (as well
as the witness for $\tk \in R^{TAUT}$ required in the general formulation of the
theorem). Such an algorithm is then polynomial in the size of $\tk$.
If $\tau$ is a solution to $\mbox{{\bf Cert}}(c)$ on input $D_k$ then either 
$\tau \in \tk$ and $\forall y (|y| \le k^c) \neg F_1(\tau, y, w_k)$
or $\tau \notin \tk$ and $D_k(\tau, y)$ is satisfiable, in which case
$\tpv$ implies that
$\forall y (|y| \le k^c) \neg F_0(\tau, y, w_k)$. In other words, 
an algorithm solving $\mbox{{\bf Cert}}(c)$
solves $\mbox{{\bf Err}}(TAUT, {\cal F})$ too. This yields the following statement
as a corollary to Theorem \ref{4.1}.

\begin{corollary}
Assume that an exponentially hard one-way permutation exists.
Then there is $c \geq 1$ such that no deterministic time $2^{O(k)}$
algorithm can solve $\mbox{{\bf Cert}}(c)$ on all lengths $k \geq 1$.
\end{corollary}

\section{Proof systems with advice and with extra axioms}
\label{5}

The task {\bf Find} was formulated using the provability in a proof system
and in this section we develop a technical tool allowing us to move
from $\npcpp$ algorithms to proof systems.
We shall recall first the notion of 
a {\bf proof system with advice} as introduced by Cook-K.\cite[Def.6.1]{CooKra}. 
It is defined as the ordinary Cook-Reckhow proof system (cf.the introduction)
except that the binary relation $P(x,y)$ is decidable in polynomial time
using an advice string that depends only on the length of $x$ (the formula). 
We say that the advice
is polynomial iff its length is $|x|^{O(1)}$.
This concept
has some interesting properties; for example, in the classes of these proof systems
- with varying bounds on the size of advice strings - there exists an optimal one.
We refer the reader to \cite[Sec.6]{CooKra} and to subsequent \cite{BKM1,BKM2,BeyMul}
for further information.

Our aim in this section is to link proof systems with polynomial advice 
with proof systems with extra axioms, as defined below.
A sequence of formulas $\{\alpha_k\}_k$ will be called {\bf p-bounded} iff
$|\alpha_k| \le k^{O(1)}$ for all $k$.

\begin{definition}\label{5.1}
Let $P(x,y)$ be an ordinary Cook-Reckhow proof system.

\begin{enumerate}

\item For a tautology $\alpha$ the proof system $P+\alpha$ is defined as follows:

a string $\pi$ is a $(P+\alpha)$-proof of formula $\tau$ iff
$\pi$ is a $P$-proof of a disjunction of the form
\[
\bigvee_i \neg \alpha'_i\ \vee \ \tau
\]
where $\alpha'_i$ are arbitrary substitution instances of $\alpha$
obtained by substituting constants and variables for variables.

\item
For a p-bounded sequence $\{\alpha_k\}_k$ of tautologies define a string $\pi$ to
be an $(P+ \{\alpha_k\}_k)$-proof of formula $\tau$ iff it is a $(P+\alpha_k)$-proof
of $\tau$ for $k = |\tau|$.

\end{enumerate}
\end{definition}
We allow only substitutions of constants and variables in instances $\alpha_i'$
in part 1 as that makes sense for all proof systems (e.g. we do not have to discuss
various limitations on depth for constant depth Frege systems) and it suffices here.
Systems $(P+ \{\alpha_k\}_k)$ are not meant to genuinely formalize the informal
notion of proof systems with extra axioms; such systems should not pose
restrictions on which extra axioms can be used in proofs of which formulas. We
use them here only as a technical vehicle allowing us to move 
from proof systems with advice
to ordinary proof systems.

\bigskip

Note that while $P+\alpha$ is a Cook-Reckhow proof system,  
$P+ \{\alpha_k\}_k$ is generally not.
The following lemma is obvious as we can use the sequence $\{\alpha_k\}_k$
as advice strings to recognize $(P+ \{\alpha_k\}_k)$ - proofs.

\begin{lemma}\label{5.2}
Let $P$ be a Cook-Reckhow proof system. For every p-bounded sequence
$\{\alpha_k\}_k$
of tautologies $P+ \{\alpha_k\}_k$ is a proof system with polynomial advice
in the sense of Cook-K.\cite{CooKra}.

\end{lemma}

In Section \ref{tasks} we used informally the notion of a decent proof system,
meaning a proof system that can perform efficiently a few simple manipulations with proofs. 
We shall use the formalization of this notion from \cite[Sec.2]{Kra-sat}.

In the following $\sat(u,x,v)$ are formulas 
for $k \geq 1$ and suitable $r = k^{O(1)}$ 
with 
$u = (u_1,\dots, u_k)$, $x = (x_1,\dots, x_k)$ and
$v = (v_1, \dots, v_r)$ such that
for all $a, \varphi \in \kk$ it holds that: 
\begin{itemize}
\item 
{\em $\sat(a,\varphi, v) \in \mbox{TAUT}$ iff
$a$ is a truth assignment satisfying formula $\varphi$.}
\end{itemize}
The extra variables $v$ are used to compute the truth value, as in the $\np$-completeness
of SAT.

A proof system (ordinary or with advice) 
$P$ is {\bf decent} iff the following tasks
can be performed by polynomial time algorithms\footnote{Polynomially
bounded functions would suffice for us here but such a weakening would not
put more of the usual proof systems into the class and so we just stick with the definition
from \cite{Kra-sat}.}:

\begin{enumerate}

\item[D1] From a $P$-proof $\pi$ of formula $\psi(x)$
and a truth assignment $a$ to variables $x$
construct a $P$-proof of $\psi(a)$.

\item[D2]
Given a true sentence $\psi$ (i.e. no variables)
construct its $P$-proof.

\item[D3]
Given $P$-proofs $\pi_1$ of $\psi$ and $\pi_2$ of 
$\psi \rightarrow  \eta$ construct a proof of $\eta$.

\item[D4] 
Given a formula $\varphi(u_1,\dots, u_n)$
and a $P$-proof of formula $\sat(u, \varphi, v)$
with variables $u, v$
construct a $P$-proof of $\varphi$.

\end{enumerate}
Conditions D1-3 are easy to verify for many of the usual proof
systems (e.g. Frege systems mentioned in the introduction or resolution).
The algorithm for condition D4 is defined by
induction on the number of connectives in $\varphi$, cf.\cite[Chpt.9]{kniha}.

\begin{lemma}\label{5.3}
Let $Q$ be a proof system with polynomial advice and $P$ a decent 
Cook-Reckhow proof system.
Then for every constant $c \geq 1$ there exists
a p-bounded sequence $\{\alpha_k\}_k$ of tautologies and $d \geq 1$
such that:

Any tautology $\tau$ having a $Q$-proof of size $\le |\tau|^c$
has an $(P+ \{\alpha_k\}_k)$-proof of size $\le |\tau|^d$.

\end{lemma}

\begin{proof}

Assume that $Q$ uses polynomial size advice string $w_k$ for formulas of size $k$.
For $k \geq 1$ denote by $\prov(x,y,t,s)$ a propositional formula 
such that:
\begin{itemize}

\item
$\prov$ has
$k$ atoms $x=(x_1,\dots,x_k)$ for bits of a formula, $k^c$ atoms $y = (y_1,\dots,y_{k^c})$
for bits of a $Q$-proof,
$k^c$ atoms $t = (t_1, \dots, t_{k^c})$
for bits of an advice and $k^{O(1)}$ atoms $s = (s_1, \dots, s_{k^{O(1)}})$
for bits of the computation
of the truth value of $Q(x,y)$ with advice $t$, 

\item For size $k$ formula $\varphi$ and a $k^c$ size strings $\pi$ and $w$:
$\prov(\varphi, \pi, w, s) \in \mbox{SAT}$ 
iff $Q(\varphi,\pi)$ with advice $w$ is true.
\end{itemize}
Then take for $\alpha_k$ the formula with variables $x,y,z,s,v$
\begin{equation}\label{m0} 
\prov(x,y,t/w_k,s)\ \rightarrow\ \sat(x,z,v)\ 
\end{equation}
where $w_k$ is the string used by $Q$ for size $k$
formulas. The formula expresses the soundness of
$Q$ and hence it is a tautology. 
Its total size is $k^{O(c)}$.

\bigskip

Let $\varphi$ be a size $k$ formula with variables
among $z=(z_1,\dots,z_k)$ and
having a size $\le k^c$ $Q$-proof
$\pi$. Let $e$ be bits of the computation of $Q(\varphi,\pi)$ with
advice $w_k$.

Take the following substitution instance of $\alpha_k$:
\begin{equation}%\label{m1}
\prov(\varphi, \pi, w_k, e) \ \rightarrow\ \sat(\varphi, z, v)
\end{equation}

\medskip
\noindent
{\bf Claim:}{\em 
For some $d \geq 1$ depending only on $c \geq 1$ and $P$ the formula
$\varphi$ has a $(P+\alpha_k)$-proof of size $\le k^d$.}

\medskip

We shall use the decency of $P$. By the choice of
$\pi$ and $e$ the sentence $\prov(\varphi, \pi, w_k, e)$ 
is true and hence has, by D2, a size $k^{O(c)}$ $P$-proof.
Then, using D3, use modus ponens to derive in size $k^{O(c)}$
the formula
\[
\sat(\varphi, z, v)\ .
\] 
D4 then allows to derive in $P$ the formula $\varphi$, in size $k^{O(1)}$.
The total size of the $P$-proof is $k^{O(c)}$.

\end{proof}

Note that we have not used the decency condition D1 explicitly;
it's role is replaced here by the definition of the system $P + \alpha_k$
which takes as axioms all substitution instances of $\alpha_k$.

Formulas $\alpha_k$ depend not only on $k$ and $w_k$ but also on the bound $k^c$ to
the length of $y$. This is the reason why we cannot simply say that
$P + \{\alpha_k\}_k$ simulates $Q$ (it simulates only\footnote{$P$ augmented by all
such formulas $\alpha_k$, constructed for all $k$, $w_k$ and all $c \geq 1$, does
simulate $Q$ as by the proof of the claim $d = O(c)$.}
$Q$-proofs of size
at most $k^c$).

\section{The partial definability of tasks {\bf Cert} and {\bf Find}}
\label{6}

We will need the following notion.
For a complexity class $\mathcal X$ and a language $L$ define
the property that {\bf $L$ is infinitely often in ${\mathcal X}$},
denoted $L \in_{i.o.} {\mathcal X}$, iff there exists $L'\in {\mathcal X}$
such that
\[
L \cap \kk\ =\ L' \cap \kk
\]
for infinitely many lengths $k$.
Recall the definition of the class $\npcpp$ at the beginning of
Section \ref{4}.

\bigskip

The following consequence of Conjecture \ref{2.1} 
was noted at the end of \cite[Sec.30.2]{k2} and it
uses an idea linking the output/input ratio of proof complexity
generators with the unprovability of circuit lower bounds
due to Razborov \cite{Raz95}, quite similar to the reasoning 
in Razborov-Rudich \cite{RR}.

\begin{lemma}\label{2.3}
\hfill

Assume that Conjecture \ref{2.1} holds
and that an exponentially hard one-way permutation exists.
Then for every $L \in \neco$:
\[
L \in_{i.o.}  \npcpp\ .
\]
In particular, $TAUT \in_{i.o.} \np/poly$.
\end{lemma}

\begin{proof}

Take $\delta > 0$ from Conjecture \ref{2.1}.
Put $k := n^{\delta}$ and think of a string $b \in \mm$ as
of the truth-table of the characteristic function of the language $L$ 
on inputs of length $k$; denote it $L_k$ as in Section \ref{4}.

For any language $L$ in $\neco$ the set of strings 
$\{L_k\ |\ k \geq 1\}$
is in $\np$: the $\np$ witness can collect all $2^k$ $\nee$ witnesses
for each variable setting - this will have size $2^{O(k)}$ -
and check their validity.

In particular, if some $L \in \neco$ would determine the
truth-tables $L_k$ for $k = n^{\delta}$ such that
all but finitely many lie outside the range of $\nw$ we would get
a contradiction with Conjecture \ref{2.1}. Hence we get:

\medskip
\noindent
{\bf Claim:} {\em For infinitely many $n$, for $k = n^{\delta}$ and $m = 2^k$:
\[
L_k \in \mm \cap Rng(\nw)\ .
\]}
For $L_k \in Rng(\nw)$ let $a \in \nn$ be such that
$\nw(a) = L_k$. Then computing  $L$ on $i \in \kk$ amounts to computing
$f$ on $a(J_i)$.
But by the requirement (\ref{e2}) posed on $A_n$ the set 
$a(J_i)$ can be computed from $i$ and $a$ 
(taken as the advice string)
in time polynomial in $n$
and $f$ is an $\np \cap \mbox{co}\np$ function.

\end{proof}

This lemma has an immediate consequence for problem {\bf Cert}.

\begin{corollary}
Assume that Conjecture \ref{2.1} holds
and that an exponentially hard one-way permutation exists.

Then for some $c \geq 1$ the task $\mbox{{\bf Cert}}(c)$ 
is only partially defined
for infinitely many lengths $k \geq 1$,
i.e. there are inputs corresponding to $k$ for which
the problem has no solution.

\end{corollary}

We shall derive the same consequence for the task {\bf Find}
using the results of Section \ref{5}.
For a triple $\cal F$ as in (\ref{h1}) define 
a proof system with polynomial advice $Q(x,y)$ by:
\begin{itemize}
\item either $0 < |w|\le |x|^c \wedge  F_1(x,y,w)$,

\item or $w$ is the empty word and $y$ is a Frege proof of $x$,
\end{itemize}
thinking of $w$ as the advice. 
Note that condition (\ref{h1}) does not guarantee that $F_1(x,y,w)$
is a proof system for some $w$ but the second clause of the definition
of $Q(x,y)$
allows us to fall back on a Frege system by taking for
the advice the empty word.

Now let $\{w_k\}_k$ such that $|w_k| \le k^c$ be a sequence of advice words
defining a proof system with advice $Q$ (there exists at least one such:
the sequence of empty strings).
Let $P$ be any decent Cook-Reckhow proof system and let formulas $\alpha_k$
and constant $d \geq 1$ be those provided by Lemma \ref{5.3} for $c$ from $\cal F$.
Assume $c_0 \geq 1$ is such that $|\alpha_k|\le k^{c_0}$ and assume also w.l.o.g.
that $d \geq c_0$. 

Consider the task $Find(P,c_0,d)$. A solution for input $1^{(k)}$ and $\alpha_k$
is a size $k$ tautology $\beta$ and by the choice of $c_0, d$ it must be that
\[
\forall y (|y| \le k^c)\ \neg F_1(\beta, y, w_k)\ .
\]
But Conjecture \ref{2.1} implies analogously as above that for the triple $\cal F$
coming from $\nw$ there will be infinitely many lengths $k\geq 1$
and strings $w_k$ for which $\alpha_k$ is a tautology but no such 
$\beta$ exists. Hence we have the following statement.

\begin{theorem}
Assume that Conjecture \ref{2.1} holds
and that an exponentially hard one-way permutation exists.

Then for all decent Cook-Reckhow proof systems $P$ there are
constants $c_1 \geq c_0 \geq 1$ such that the task 
$\mbox{{\bf Find}}(P, c_0, c_1)$ has
no solution for infinitely many lengths $k \geq 1$.

\end{theorem}

\section{Disjoint $\np$ pairs}
\label{pairs}

Let $(U,V)$ and $(A,B)$ be two pairs of disjoint subsets of $\uu$.
A {\bf reduction} of $(A,B)$ to $(U,V)$ is a function $f : \uu \rightarrow \uu$
such that for all $u$:
\[
u \in A \rightarrow f(u) \in U\ \wedge\ 
u \in B \rightarrow f(u) \in V\ .
\]
It is (non-uniform) {\bf p-reduction} if $f$ is (non-uniform) p-time.

Disjoint $\np$ pairs appear quite naturally in several places of 
proof complexity. Most notably, the reflection principle for a
proof system just asserts that two $\np$ sets (that of formulas with
bounded size proofs and of falsifiable formulas) are 
disjoint. A particularly elegant form of this observation was found by Razborov 
\cite{Raz-pairs} in the notion of the canonical pair of a proof system.
Shadowing the relation of the provability of reflection principles to simulations,
a similar one exists between the provability of disjointness of such pairs and
simulations. 
We refer the reader to \cite{Pud-survey} for more background.

Given two pairs of disjoint sets
$(U,V)$ and $(A,B)$ and a constant $c \geq 1$
consider the {\bf search task} $\mbox{{\bf Pair}}_{U,V}^{A,B}(c)$:

\begin{itemize}

\item input: $1^{(k)}$ and
a circuit $C$ with $k$ inputs, several outputs
and of size at most $k^c$

\item required output: a string $u \in \kk$ such that
\[
u \in A \wedge f(u) \notin U\ \mbox{ or }\  
u \in B \wedge f(u) \notin V\ .
\]

\end{itemize}
In other words, the output string $u$ certifies that circuit 
$C$ is not a reduction
of $(A, B)$ to $(U,V)$ on $\kk$.

\bigskip

Take a triple $\cal F$ as in (\ref{h1}) and define $U$ and $V$
to be the sets of pairs $(x,z)$ such that $|z|\le |x|^c$ and
$\exists y (|y|\le |x|^c)F_0(x,y,z)$ or 
 $\exists y (|y|\le |x|^c)F_1(x,y,z)$, respectively.
 
For a disjoint pair $A, B$ of sets such that $A \in \neco$ take for language
$L$ on $\kk$ simply $A$. 
For $w$ of size $\le k^c$ consider circuit $C_w$ that takes size $k$ input $x$
and outputs the pair $(x,w)$; note that $|C_w| \le k^{c+1}$ for $k >> 0$.
Then a solution to 
$\mbox{{\bf Pair}}_{U,V}^{A,B}(c+1)$ for input $1^{(k)}$ and
$C_w$ is also a solution to 
$Err(L,{\cal F})$. Hence Theorem \ref{4.1} implies the following statement.

\begin{theorem}
Assume that an exponentially hard one-way permutation exists.

Then there are two disjoint $\np$ sets $U, V$ and $c \geq 1$ such that for any pair
$A, B$ of disjoint sets such that 
$A \in \neco$ the task $\mbox{{\bf Pair}}_{U,V}^{A,B}(c)$
is not solvable by a deterministic time $2^{O(k)}$ algorithm.

\end{theorem}

The reader familiar with the canonical pairs of proof systems mentioned earlier
may note that these sets are in $\ee \subseteq \neco$ and thus the theorem applies to them.

\section{Concluding remarks}
\label{rema}

The role of Conjecture \ref{2.1} is rather ambivalent: it implies
that $\np \neq co\np$ (Lemma \ref{2.2}) but also 
that $TAUT \in _{i.o.} \np/poly$ (Lemma \ref{2.3}).
This is caused by the dual role of the Nisan-Wigderson generator; it is a source of
hard tautologies but also a strong proof system.
The reader should consider, before dismissing Conjecture \ref{2.1}
because of Lemma \ref{2.3}, how little contemporary complexity theory understands
about the power of non-uniform advice.

It would be interesting to have a variety of candidate combinatorial 
constructions $Q(P)$ of a proof system stronger than $P$, as discussed
in Section \ref{tasks}. At present only the construction of $iP$, the implicit $P$,
from \cite{Kra-implicit} applies to all proof systems and it is consistent
with the present knowledge that it indeed yields stronger proof systems.
An indirect plausible construction of $Q(P)$ may use the relation between proof systems
and first-order theories: take theory $T_P$ corresponding to $P$,
extend $T_P$ by $Con(T_P)$ (or in some other G\"{o}delian fashion)
getting $S$, and then take for $Q(P)$ the proof system simulating $S$
(cf.\cite{KP-jsl,kniha} about $T_P$ etc.).
But it is hardly combinatorially transparent.

The referee pointed out that problems $Cert$ and $Find$ seem to be very close
to each other and asked whether they could be, in fact, reduced to each
other. One can reduce $Find(P, c_0, c_1)$ to $Cert(c)$ for suitable $c$:
for $k \geq 1$ and a tautology $\alpha$ of size at most $k^{c_0}$
take for $D$ a circuit (canonically constructed) computing the
provability relation for $P + \alpha$ on formulas of size $k$ and
proofs of size at most $k^{c_1}$. The size of $D$ is polynomial in $k^{c_1}$
and so $D$ is a valid input to $Cert(c)$ for some $c_1 \le c = O(c_1)$.
A solution to $Cert(c)$ cannot certify that $D$ is unsound
and so it has to be a size $k$ formula certifying the
lower bound $k^c \geq k^{c_1}$
for $D$ and hence the $k^{c_1}$ lower bound for $P + \alpha$ too.
I do not know whether $Cert$ can be also reduced to $Find$ but I suspect it cannot
be: valid inputs to $Cert$ are allowed either to fail to be a sound
proof system or to fail to have short proofs while inputs to $Find$ 
cannot be unsound. 

\begin{acknowledgements}
I thank the anonymous referee for comments that lead to a better presentation
of the results.

\end{acknowledgements}

\affiliationone{Department of Algebra\\
Faculty of Mathematics and Physics\\
Charles University\\
Sokolovsk\' a 83, Prague 8\\ 
CZ - 186 75, The Czech Republic
\email{krajicek@karlin.mff.cuni.cz}}

\end{document}